\documentclass[12pt,a4paper]{amsart}

\usepackage{amsmath}
\usepackage{amsfonts}
\usepackage{amssymb}
\usepackage{verbatim}
\usepackage{amsthm}
\usepackage{color}
\usepackage{xypic,mathtools}
\usepackage{hyperref}
\usepackage{tikz}
\usetikzlibrary{matrix,arrows,decorations.pathmorphing,positioning}
\usepackage[left=2cm,right=2cm,top=2cm,bottom=2cm]{geometry}

\theoremstyle{plain}
 \newtheorem{theorem}{Theorem}
 \newtheorem{lemma}{Lemma}

\theoremstyle{definition}

\DeclareMathOperator{\disc}{disc}

\DeclareMathOperator{\Gal}{Gal}

\DeclareMathOperator{\Gr}{Gr}

\DeclareMathOperator{\lcm}{lcm}

\DeclareMathOperator{\Pic}{Pic}

\DeclareMathOperator{\SL}{SL}

\newcommand{\IZ}{\mathbb{Z}}

\newcommand{\IZhat}{\widehat{\IZ}}

\newcommand{\IQ}{\mathbb{Q}}

\newcommand{\IQbar}{\overline{\mathbb{Q}}}

\newcommand{\IC}{\mathbb{C}}
\newcommand{\IF}{\mathbb{F}}

\newcommand{\IP}{\mathbb{P}}
\newcommand{\IA}{\mathbb{A}}

\newcommand{\house}[1]{\lceil #1 \rceil}

\title{Linear Equations in Singular Moduli}

\author{Yuri Bilu}
\email{yuri@math.u-bordeaux.fr}
\address{IMB, Universit\'{e} de Bordeaux \\  
351, cours de la Libération \\
33 405 Talence cedex \\
France}

\author{Lars K\"uhne}
\email{lars.kuehne@unibas.ch}
\address{Departement Mathematik und Informatik \\
Spiegelgasse 1 \\
4051 Basel \\
Switzerland}

\subjclass[2010]{11G18 (primary), 11G50, 14G35, 11R37} 

\begin{document}

\begin{abstract}
We establish an effective version of the Andr\'e-Oort conjecture for linear subspaces of $Y(1)^n_\IC \approx \IA_\IC^n$. 
Apart from the trivial examples provided by weakly special subvarieties, this yields the first algebraic subvarieties in a Shimura variety of dimension $> \! 1$ whose CM-points can be (theoretically) determined.
\end{abstract}

\maketitle

\section{Introduction}

In a previous article \cite{Kuehne2017}, one of us obtained restrictions on the intersection of ring class fields associated with distinct imaginary quadratic fields. In addition, he was able to reprove a weak corollary of the Andr\'e-Oort conjecture effectively (\cite[Theorem 3]{Kuehne2017}). Unfortunately, his results fall short of establishing complete results of Andr\'e-Oort type effectively. In this article, we complement the technique of \cite{Kuehne2017} and prove effective Andr\'e-Oort type results for a class of subvarieties in products of modular curves. The novelty is that this class contains varieties of arbitrary high dimension. 

Up to now, all non-trivial effectively solvable cases (\cite{Allombert2015, Bilu2016, Bilu2013, Kuehne2012, Kuehne2013}) of the Andr\'e-Oort conjecture have been restricted to the case of curves. In particular, the only known examples of algebraic subvarieties in Shimura varieties that are known to contain no special points are either curves or weakly special subvarieties. That the latter do not contain special points unless they are actually special subvarieties is a direct consequence of Moonen's characterization of special subvarieties (\cite[Theorem 4.3]{Moonen1998}).

It is a well-known fact that the $j$-invariant yields an isomorphism $Y(1) = \IA^1_\IQ$ and that this identification establishes $\IA^1_\IQ$ as a canonical $\IQ$-model of $Y(1)$. There is hence a well-defined notion of \textit{(affine) linear subvarieties} in $Y(1)^n$ ($n \geq 2$). Additionally, the positive-dimensional maximal special subvarieties (i.e., the maximal subvarieties of Hodge type) of a proper linear subspace are linear subvarieties themselves. Furthermore, any special linear subvariety of dimension $n-1$ is of the form $V(z_i = z_j) \subset Y(1)^n$ where $(i,j)$ is a pair of distinct integers from $\{ 1, \dots, n\}$. A general special linear subvariety is just an intersection of finitely many of these special linear hypersurfaces. Our Section \ref{subsection::specialsubvarieties} provides a proof of these facts.

An immediate consequence of this description of the positive-dimensional maximal special subvarieties in $L$ is that there are at most finitely many of them and that they can be easily determined effectively. Consequently, the main difficulty in proving the Andr\'e-Oort conjecture for linear subvarieties is with those special points not contained in any positive-dimensional special subvarieties. 

To state our main result, we introduce some notions related to special points (for details see Section \ref{subsection::specialpoints}). The components of a special point in $Y(1)^n$ are CM-points in $Y(1)$. In the moduli interpretation, such a CM-point corresponds to an isomorphism class of elliptic curves whose endomorphism ring is an imaginary quadratic order. These imaginary quadratic orders can be uniquely described in terms of their discriminants. In this way, we associate with each special point $P \in Y(1)^n(\IQbar)$ a $n$-tuple $\Delta(P)=(\Delta_1(P),\cdots,\Delta_n(P))$ of such discriminants. Since there are only finitely many CM-elliptic curves whose endomorphism ring has a bounded discriminant, bounding discriminants for a set of CM-points amounts to proving its finiteness. With these preparations, we can formulate our main result as follows. The height $H(L)$ of a linear subvariety $L \subseteq Y(1)^n_{\IQbar} = \IQbar^n$ is defined in Section \ref{subsection::heights}.

\begin{theorem} \label{theorem::main} Let $N$ be a normal number field and $L \subseteq Y(1)^n_{N}$ a linear subvariety of height $H(L)$. Denote by $Z^{\mathrm{sp}}$ the union of the positive-dimensional special subvarieties contained in $L$. With constants
\begin{equation*}
c_1 = 480 n^264^n[N:\IQ]^3,
\end{equation*}
and
\begin{equation*}
c_2 = (1.4 \cdot 10^{11}) (2.1 \cdot 10^4)^{n} (n+1)^{4n+6} [N:\IQ]^4,
\end{equation*}
we have
\begin{equation*}
\max \{ |\Delta_1(P)|, \dots, |\Delta_n(P)| \}^{1/2} \leq c_1 \cdot \log(H(L)) + c_2
\end{equation*}
for any special point $P \in (L \setminus Z^{\mathrm{sp}})(\IQbar)$.
\end{theorem}

In \cite{Pila2011a}, it is already shown that there are only finitely many CM-points in $L \setminus Z^{\mathrm{sp}}$. The new information in our above theorem is the explicit bound on the discriminants that allows to compute them, at least theoretically. By using our proof as an algorithm to determine all special points outside of $Z^{\mathrm{sp}}$ rather than to obtain the above general bounds, it seems possible to find situations where $L \setminus Z^{\mathrm{sp}}$ is non-empty but contains no special points. In contrast, Pila's proof in \cite{Pila2011a} blends point counting in o-minimal structures, the Pila-Wilkie theorem \cite{Pila2006}, with lower bounds on Galois orbits, deduced from Siegel's bound on class numbers of imaginary quadratic fields \cite{Siegel1935}. Neither of these tools is effective, and each one alone constitutes a serious obstruction to effectivity in his approach.

We deduce Theorem \ref{theorem::main} in Section \ref{section::andreoort} rather straightforwardly from our Lemma \ref{lemma::samefield} on linear equations in \textit{distinct} singular moduli. Since a special subvariety of a Shimura variety contains infinitely many special points (e.g., by \cite[Lemmas 13.3 and 13.5]{Milne2005}), it is actually possible to reprove the above characterization of positive-dimensional special subvarieties by this deduction. The pivotal part of our argument is the proof of Lemma \ref{lemma::samefield}, which occupies all of Section \ref{section::linearequations}.




\textbf{Notations and conventions.} All number fields are contained in a fixed algebraic closure $\IQbar$ (i.e., a number field $N$ is a finite extension $\IQ \subseteq N \subset \IQbar$). We furthermore fix an embedding $\IQbar \hookrightarrow \IC$.

\section{Preliminaries}

\subsection{Heights and Linear Subspaces}
\label{subsection::heights}

We refer to the first two sections of the textbook \cite{Bombieri2006} for basics on heights and to \cite[Lecture 6]{Harris1995} for Grassmannians. Let $N$ be a number field. We denote its finite (resp.\ infinite) places of $N$ by $\Sigma_f(N)$ (resp.\ $\Sigma_\infty(N)$). For each place $\nu \in \Sigma_f(N) \cup \Sigma_\infty(N)$, we set $c_\nu = [N_\nu:\IQ_\nu]/[N:\IQ]$.

Given a point $p = (p_0:p_1:\cdots:p_n) \in \IP^n(N)$, we define projective Weil heights by 
\begin{equation*}
H^{(2)}(p) = H_{f}(p) \cdot H_{\infty}^{(2)}(p) \text{ and } H^{(\infty)}(p) = H_{f}(p) \cdot H_{\infty}^{(\infty)}(p)
\end{equation*}
where
\begin{equation*}
H_{f}(p) =
\prod_{\nu \in \Sigma_f(N)} \max \{|p_0|_\nu,|p_1|_\nu,\dots,|p_n|_\nu \}^{c_\nu},
\end{equation*}
\begin{equation*}
H_{\infty}^{(\infty)}(p) =
\prod_{\nu \in \Sigma_\infty(N)} \max \{|p_0|_\nu,|p_1|_\nu,\dots,|p_n|_\nu \}^{c_\nu},
\end{equation*}
and
\begin{equation*}
H_{\infty}^{(2)}(p) =
\prod_{\nu \in \Sigma_\infty(N)} \left( |p_0|_\nu^2+ |p_1|_\nu^2+ \cdots +|p_n|_\nu^2 \right)^{c_\nu/2}.
\end{equation*}
The definition of both $H^{(2)}(p)$ and $H^{(\infty)}(p)$ does not depend on the choice of a number field $N$ such that $p \in \IP^n(N)$. 

The projective height also induces an affine Weil height by setting
\begin{equation*}
H(p)=H^{(\infty)}(1:p_1:\cdots:p_n)
\end{equation*}
for any $p=(p_1,p_2,\dots,p_n) \in \IA^n(\IQbar)$. In our arguments, we frequently use the standard inequality
\begin{equation*}
H(p+q) \leq 2H(p)H(q)
\end{equation*}
for any $p,q \in \IA^n(\IQbar)$ (see \cite[Proposition 1.5.15]{Bombieri2006}), as well as Liouville's inequality
\begin{equation*}
|\alpha|_\nu \geq H(\alpha)^{-[N:\IQ]}
\end{equation*}
for any $\alpha \in N^\times \subset \IA^{1}(N)$ and any place $\nu \in \Sigma_f(N) \cup \Sigma_\infty(N)$ (see \cite[1.5.19]{Bombieri2006}).

We next define the height $H(L)$ of an $l$-dimensional linear subspace $L \subseteq \IQbar^n$. To such a linear variety is associated a point $p_L \in \Gr(n,l)(\IQbar)$. Grassmann coordinates determine an embedding $\Gr(n,l)_{\IQbar}\hookrightarrow \IP(\bigwedge^l \IQbar^n)$ and via standard bases we identify $\IP(\bigwedge^l \IQbar^n)$ with $\IP^{N}_{\IQbar}$, $N=\binom{n}{l}$. Using the projective Weil height defined above, we simply set $H(L)=H^{(2)}(p_L)$. In the terminology of \cite[Definition 2.8.5]{Bombieri2006}, we have $H(L)=\exp(h_{\mathrm{Ar}}(L))$. In Section \ref{section::andreoort}, we need a more explicit formula for this height. For this, let $\{ \underline{b}_1, \dots, \underline{b}_l \} \subset \IQbar^n$ be a basis of $L$ so that we have a matrix
\begin{equation*}
B = (\underline{b}_1, \dots, \underline{b}_l) = 
\begin{pmatrix}
b_{11} & b_{12} & \cdots & b_{1l} \\
b_{21} & b_{22} & \cdots & b_{2l} \\
\cdots & \cdots & \cdots & \cdots \\
b_{n1} & b_{n2} & \cdots & b_{nl}
\end{pmatrix}
\in \IQbar^{n \times l}
\end{equation*}
With each subset $I = \{ i_1, \dots, i_l\} \subseteq \{ 1, \dots, n \}$ of cardinality $l$, we associate an $(l \times l)$-minor
\begin{equation*}
B_I = 
\begin{pmatrix}
b_{i_11} & b_{i_12} & \cdots & b_{i_1l} \\
b_{i_21} & b_{i_22} & \cdots & b_{i_2l} \\
\cdots & \cdots & \cdots & \cdots \\
b_{i_l1} & b_{i_l2} & \cdots & b_{i_ll}
\end{pmatrix}
\end{equation*}
Unravelling definitions, we have
\begin{equation} \label{equation::heightformula}
H(L) = 
\prod_{\nu \in \Sigma_f(N)} \left( \max_I \left\lbrace \left\vert\det(B_I)\right\vert_\nu\right\rbrace \right)
\prod_{\nu \in \Sigma_\infty(N)} \left( \sum_I\left\vert\det(B_I)\right\vert_\nu^2
\right)^{1/2}
\end{equation}
where, in both the maximum and the sum, $I$ runs through the subsets of $\{ 1, \dots, n \}$ having cardinality $l$ (see \cite[Remark 2.8.7]{Bombieri2006}).

For Theorem \ref{theorem::main}, we need to extend this height to general (i.e., inhomogeneous) linear subvarieties $L \subseteq \IA^n_{\IQbar}$. With such a linear subvariety is associated its homogenization $L^{h} \subseteq \IP_{\IQbar}^n$ (i.e., its Zariski closure in $\IP_{\IQbar}^n$). Its preimage $\pi^{-1}(L^{h})$ under the standard projection $\pi: \IA^{n+1}_{\IQbar} \rightarrow \IP^n_{\IQbar}$ is a homogeneous linear subvariety. If $L$ is a linear subspace of $\IA^n_{\IQbar}$, it is easy to check that $H(L)=H(\pi^{-1}(L^{h}))$. For a general linear subvariety, we simply set $H(L)=H(\pi^{-1}(L^{h}))$, extending the previous definition.

We conclude with a simple observation: For every linear subvariety $L \subsetneq N^n$ of dimension $l$ there exists a hyperplane $L^\prime \subset N^n$ such that $L \subseteq L^\prime$ and $H(L^\prime) \leq H(L)$. Inspecting (\ref{equation::heightformula}), one can see that this means there exists a non-trivial linear equation
\begin{equation} \label{equation::linearequationheights}
a_1z_1 + \cdots + a_nz_n + b= 0, \ a_1,\dots,a_n,b \in N,
\end{equation}
with 
\begin{equation*}
H^{(\infty)}(a_1:\dots: a_n: b) \leq H^{(2)}(a_1:\dots: a_n: b) \leq H(L)
\end{equation*}
such that any point $(z_1,z_2,\dots,z_n) \in L(\IQbar)$ is a solution of (\ref{equation::linearequationheights}). Scaling $(a_1,\dots,a_n,b)$ by a non-zero constant, we may even assume that $H(a_1,\dots,a_n,b) \leq H(L)$. It is easy to see that it suffices to prove the assertion in the homogeneous case. Let $\{ \underline{b}_1, \underline{b}_2, \dots, \underline{b}_{l} \} \subset N^n$ be a basis of $L$. Denote by $\{ \underline{e}_1, \underline{e}_2, \dots, \underline{e}_{n} \} \subset N^n$ the standard basis of $N^n$ (i.e., those vectors whose one component is $1$ and whose other components are $0$). If $l< n-1$, there exist vectors $\underline{e}_{i_1}, \dots, \underline{e}_{i_{n-1-l}}$ such that the span $\left\langle \underline{b}_1, \underline{b}_2, \dots, \underline{b}_{l}, \underline{e}_{i_1}, \dots, \underline{e}_{i_{n-1-l}} \right\rangle$ is a hyperplane $L^\prime \subset N^n$. From \cite[Remark 2.8.9]{Bombieri2006}, we know that 
\begin{equation*}
H(L^\prime) \leq H(L)H(\langle \underline{e}_{i_1} \rangle) \cdots H(\langle \underline{e}_{i_{n-1-l}}\rangle) \leq H(L).
\end{equation*}

\subsection{Special points on $Y(1)^n$} \label{subsection::specialpoints}

Our basic reference on special points is \cite{Cox1989}; the reader is also referred to \cite[Sections 2.2 and 2.3]{Kuehne2017} for an brief summary in Deligne's terminology. By a CM-period we mean a point $\tau \in \mathcal{H} = \{ z \in \IC \ | \ \mathrm{Im}(z) > 0\}$ such that $[\IQ(\tau):\IQ]=2$. For each CM-period $\tau$, the endomorphism ring of the lattice $\IZ[\tau]$ is an order $\mathcal{O}(\tau)$ in the imaginary quadratic field $\IQ(\tau)$. We write $\Delta(\tau)$ for the discriminant of $\mathcal{O}(\tau)$, which we also call the discriminant of $\tau$ in the sequel. Writing $f$ for the conductor of $\mathcal{O}(\tau)$ with respect to the integer ring $\mathcal{O}_{\IQ(\tau)}$ of $\IQ(\tau)$, we have $\Delta(\tau)=f^2 \disc(\mathcal{O}_{\IQ(\tau)})$. 

 Consider the action of $\SL_2(\IZ)$ on the complex upper half plane $\mathcal{H} = \{ z \in \IC \ | \ \mathrm{Im}(z) > 0\}$ given by $\gamma \tau = \frac{a\tau+b}{c\tau+d}$ for $\gamma = \left(\begin{smallmatrix} a & b \\ c & d\end{smallmatrix}\right) \in \SL_2(\IZ)$. It is easy to check that for each CM-period, the elements of its $\SL_2(\IZ)$-orbit are exactly the CM-periods having the same discriminant. Additionally, the set
\begin{equation*}
\mathcal{F} = \{ \tau \in \mathcal{H}^+ \ | \ -\frac{1}{2}\leq  \mathrm{Re}(\tau) < \frac{1}{2} \text{ and } |\tau| > 1 \} \cup \{ \tau \in \mathcal{H}^+ \ | \ |\tau|=1 \text{ and } \mathrm{Re}(\tau)\leq 0\}
\end{equation*}
is a fundamental domain for the $\SL_2(\IZ)$-action; this means that each $SL_2(\IZ)$-orbit contains a unique representative in $\mathcal{F}$. These representatives have the form
\begin{equation} \label{equation::taus}
\tau = \frac{-b+i\sqrt{4ac-b^2}}{2a}, \ a,b,c \in \IZ,
\end{equation}
with
\begin{equation*}
-a < b \leq a < c \text{ or } 0 \leq b \leq a = c,
\end{equation*}
and $\Delta(\tau)=b^2-4ac<0$. If $\Delta<0$ denotes the discriminant of an imaginary quadratic order, we can hence choose $b_\Delta \in \{ 0, 1\}$ such that $b_\Delta \equiv \Delta \pmod 2$, and
\begin{equation} \label{equation::taudelta}
\tau_\Delta = \frac{-b_\Delta + i\sqrt{|\Delta|}}{2}
\end{equation}
is a CM-period of discriminant $\Delta$.

Klein's $j$-invariant
\begin{equation*}
j(\tau) = q^{-1} + 744 + 196884q + \cdots, \ q=e^{2\pi i \tau},
\end{equation*}
induces a bijection between the quotient $\SL_2(\IZ) \backslash \mathcal{H}^+$ and $\IC$. In fact, we can use this bijection to identify $\SL_2(\IZ) \backslash \mathcal{H}^+$ with the complex points of an algebraic curve $Y(1)$ over $\IQ$, and $Y(1)$ is isomorphic to the affine algebraic line $\IA^1_\IQ$. 

For a CM-period $\tau$, the value $j(\tau)$ is known to be algebraic and called a singular modulus. For each imaginary quadratic order $\mathcal{O}$, the polynomial
\begin{equation*}
H_{\mathcal{O}}(X)=\prod_{\substack{\tau \in \mathcal{F} \\ \Delta(\tau)=\disc(\mathcal{O})}}(X-j(\tau))
\end{equation*}
(``the class equation'') is irreducible over $\IQ$ (see \cite[Section 13]{Cox1989}). Consequently, the $\IQ$-Galois conjugates of a singular modulus $j(\tau)$ are precisely the singular moduli $j(\tau^\prime)$ with $\Delta(\tau^\prime)=\Delta(\tau)$.  In particular, each singular modulus $j(\tau)$ of discriminant $\Delta$ is $\IQ$-Galois conjugate to the singular modulus $j(\tau_{\Delta})$ where $\tau_\Delta$ is given by (\ref{equation::taudelta}).

Via the identification $Y(1)^n = \IA^n_\IQ$, an $n$-tuple $(j(\tau_1),\dots,j(\tau_n))$ of singular moduli gives rise to an algebraic point on $Y(1)^n$. We call these points the special points of $Y(1)^n$.

\subsection{Special subvarieties of linear varieties} 
\label{subsection::specialsubvarieties}

Assume that $X$, $\dim(X)\geq 1$, is a maximal special subvariety of a linear subvariety $L \subseteq Y(1)^n_\IC$. Our aim is to prove that $L$ is a linear subvariety itself. As a by-product, we also obtain that the $(n-1)$-dimensional linear special subvarieties in $Y(1)^n$ are of the form $V(z_i=z_j)$ for distinct $i,j \in \{ 1, \dots, n\}$, and that any positive-dimensional special linear subvariety is an intersection of these.

Let $\Phi_N(x,y)$ denote the $N$-th modular transformation polynomial (e.g.\ as defined in \cite[Chapter 11]{Cox1989} or \cite[Chapter 5]{Lang1987}). The partial degrees of these polynomials are
\begin{equation} \label{equation::modulardegree}
\deg_{x}(\Phi_N)=\deg_{y}(\Phi_N)=N\prod_{\substack{p \text{ prime} \\ p | N}} \left( 1 + \frac{1}{p}\right).
\end{equation}
From \cite[Section 2]{Edixhoven2005}, we know that, up to reordering coordinates, there exists a partition $n=n_0+n_1+n_2+\dots+n_r$ with positive integers $n_i$ such that 
\begin{equation} \label{equation::specialsubvariety}
X=P \times X_1 \times \cdots \times X_r \subseteq Y(1)^{n_0} \times Y(1)^{n_1} \times \cdots \times Y(1)^{n_r}
\end{equation}
where $P \in Y(1)^{n_0}$ is a special point and
\begin{equation} \label{equation::Xi}
X_i = V(\Phi_{N^{(i)}_{2}}(z_1^{(i)},z_2^{(i)}),\dots,\Phi_{N^{(i)}_{n_i}}(z_1^{(i)},z_{n_i}^{(i)})) \subseteq Y(1)^{n_i}, 1\leq i \leq r.
\end{equation}
(If $n_i=1$ this should be read as $X_i=Y(1)$.) 

We claim that $N_j^{(i)}=1$ for $1\leq i \leq r$ and $2 \leq j\leq n_i$. Since $\Phi_1(x,y)=x-y$, this is precisely what we want to show. By symmetry, it suffices to show that $N_{n_r}^{(r)} = 1$ if $n_r \geq 2$. Setting
\begin{equation*}
X_r^{\prime} = V(\Phi_{N^{(r)}_{2}}(z_1^{(r)},z_2^{(r)}),\dots,\Phi_{N^{(r)}_{n_r-1}}(z_1^{(r)},z_{n_r-1}^{(r)})) \times Y(1) \subseteq Y(1)^{n_r-1} \times Y(1)
\end{equation*}
we note that 
\begin{equation*}
P \times X_1 \times X_2 \times \cdots \times X_{r-1} \times X_r^\prime \nsubseteq L
\end{equation*}
because $X$ is maximal. This means that there exist closed points $P_{i} \in X_{i}$ ($1\leq i \leq r-1$) such that
\begin{equation}  \label{equation::notinL}
P \times P_1 \times P_2 \times \cdots \times P_{r-1} \times X_r^\prime \nsubseteq L.
\end{equation}
The curve
\begin{equation*}
C = X \cap (P \times P_1 \times \dots \times P_{r-1} \times Y(1)^{n_r})
\end{equation*}
is a special subvariety of the linear subvariety
\begin{equation*}
L^\prime = L \cap (P \times P_1 \times \dots \times P_{r-1} \times Y(1)^{n_r}),
\end{equation*}
and we consider both $C$ and $L^\prime$ simply as subvarieties of $Y(1)^{n_r}$. Let $\pi: Y(1)^{n_r} \rightarrow Y(1)^{n_r-1}$ be the projection to the first $n_r-1 \geq 1$ coordinates. Because of (\ref{equation::notinL}), the map $\pi|_{L^\prime}: L^\prime \rightarrow \pi(L^\prime)$ is finite. Consequently, it must be of degree $1$ (i.e., an isomorphism) and hence its restriction $\pi|_{C}:C \rightarrow \pi(C)$ has also degree $1$. By using (\ref{equation::modulardegree}), we infer that $1=\deg(\pi|_{C})=N_{n_r}^{(r)}\prod_{p | N_{n_r}^{(r)}} \left( 1 + 1/p\right)$ and thus $N_{n_r}^{(r)}=1$.

\subsection{Two estimates} We recall the following archimedean estimate from \cite[Lemma 1]{Bilu2013}: 
\begin{equation} \label{equation::jestimate}
\forall \tau \in \overline{\mathcal{F}}: \left\vert |j(\tau)|-e^{2\pi \mathrm{Im}(\tau)}\right\vert \leq 2079.
\end{equation}
From this, we can deduce a simple height estimate for singular moduli. Let $j(\tau)$, $\tau \in \mathcal{F}$, be a singular modulus of discriminant $\Delta$. From (\ref{equation::jestimate}), we infer that
\begin{equation} \label{equation::jhouse}
|j(\tau)| \leq e^{\pi |\Delta|^{1/2}}+2079 < 11 e^{\pi |\Delta|^{1/2}}.
\end{equation}
Since singular moduli are algebraic integers (\cite[Theorem 11.1]{Cox1989}), we have hence
\begin{equation} \label{equation::jheight}
H(j(\tau)) = \prod_{\substack{\tau \in \mathcal{F} \\ \Delta(\tau)=\disc(\mathcal{O})}} |j(\tau)|^{1/[\IQ(j(\tau)a):\IQ]} < 11 e^{\pi |\Delta|^{1/2}}.
\end{equation}

\subsection{Ring class fields} \label{subsection::ringclassfields} We refer to \cite{Cox1989} and \cite[Section 3.2]{Kuehne2017} for details on ring class fields. Let $K$ be an imaginary quadratic field and $\mathcal{O}_{K,f} = \IZ + f \mathcal{O}_{K}$ its (unique) order of conductor $f$. We let $K[f]/K$ denote the ring class field associated with $\mathcal{O}_{K,f}$. In terms of the universal norm residue symbol $(\cdot,K^{\mathrm{ab}}/K)$, $K[f]$ is the fixed field of $(\widehat{O}_{K,f}^\times K^\times/K^\times,K^{\mathrm{ab}}/K)$ so that $\Gal(K[f]/K) = \Pic(\mathcal{O}_{K,f})$ (cf.\ \cite{Kuehne2017}). 

Ring class fields are intimately related to singular moduli. Let $j(\tau)$ be the singular modulus associated with a CM-period $\tau \in K$ such that $\mathcal{O}(\tau)= \mathcal{O}_{K,f}$. Then $K(j(\tau))/K$ coincides with $K[f]$ (\cite[Theorem 11.1]{Cox1989}). In addition, the extension $K(j(\tau))/\IQ$ is Galois (\cite[Lemma 9.3]{Cox1989}).

For our main proof, we note some consequences of the well-known class number formula (\cite[Theorem 7.24]{Cox2011})
\begin{equation}
[K[f]:K] = \frac{[K[1]:K]f}{w_{K,f}} \prod_{p | f}\left( 1 - \genfrac(){}{}{\disc{(\mathcal{O}_K})}{p} \frac{1}{p}\right)
\end{equation}
with
\begin{equation}
w_{K,f} =
\begin{cases}
3 & \text{if } K = \IQ(\sqrt{-3}), f \neq 1, \\
2 & \text{if } K = \IQ(\sqrt{-1}), f \neq 1, \\
1 & \text{elsewise}.
\end{cases}
\end{equation}
Using the evident inclusion $K[c] \subset K[cf]$, we deduce that
\begin{equation*} \label{equation::classnumber_cf}
[K[cf]:K[f]]= \frac{w_{K,f}\cdot c}{w_{K,fc}} \cdot \prod_{\substack{ p \mid cf \\ p \nmid f}}\left( 1 - \genfrac(){}{}{\disc{(\mathcal{O}_K})}{p} \frac{1}{p}\right)
\end{equation*}
for any positive integers $c$ and $f$.
From this, we obtain the lower bound
\begin{equation*}
[K[cf]:K[f]] \geq \frac{c}{3} \cdot \left( \frac{1}{2} \right)^{\omega(c)}
\end{equation*}
where $\omega(n)$ is the number of prime divisors of $n$. In Section \ref{section::linearequations} below, we only need the weaker estimate
\begin{equation} \label{equation::classnumberformula}
[K[cf]:K[f]] \geq \frac{\sqrt{6}}{12} \cdot c^{1/2}.
\end{equation}

We also need some information on unions and intersections of ring class fields.

\begin{lemma} Let $K$ be an imaginary quadratic number field and $(f_1,f_2)$ a pair of positive integers. Then, we have
\begin{equation} \label{equation::rcf_intersection}
K[f_1] \cap K[f_2] = K[\gcd \{ f_1, f_2 \}]
\end{equation}
and
\begin{equation} \label{equation::rcf_union}
[K[\lcm \{ f_1, f_2 \}]: K[f_1] \cdot K[f_2]] \leq 3.
\end{equation}
\end{lemma}
\begin{proof} For a place $\nu \in \Sigma_f(K)$, we let $\mathfrak{p}_\nu$ be the associated prime ideal of $\mathcal{O}_K$ and write $p(\nu)$ for its residue characteristic. In addition, we introduce the valuation $\mathrm{val}_\nu: K_\nu^\times \rightarrow \IZ$ by setting $\mathrm{val}_\nu(x)=n \text{ if } x \in \mathfrak{p}^n_\nu \setminus \mathfrak{p}^{n+1}_\nu$. Let $I_{K,f}$ denote the finite ideles of $K$. The profinite completion
\begin{equation*}
\widehat{\mathcal{O}}_{K,f} = \{ (x_\nu)\in \prod_{\nu \in \Sigma_f(K)} \widehat{\mathcal{O}}_\nu \ | \ x_\nu \bmod{\mathfrak{p}_\nu^{\mathrm{val}_{\mathfrak{p}}(f)}} \in \IZhat_{p(\nu)} \bmod{\mathfrak{p}_\nu^{\mathrm{val}_{\mathfrak{p}}(f)}} \} \subseteq \widehat{\mathcal{O}}_K
\end{equation*}
has units
\begin{equation*}
\widehat{\mathcal{O}}_{K,f}^\times = \{ (x_\nu)\in \prod_{\nu \in \Sigma_f(K)} \widehat{\mathcal{O}}_\nu^\times \ | \ x_\nu \bmod{\mathfrak{p}_\nu^{\mathrm{val}_{\mathfrak{p}}(f)}} \in \IZhat_{p(\nu)} \bmod{\mathfrak{p}_\nu^{\mathrm{val}_{\mathfrak{p}}(f)}} \} \subseteq I_{K,f}.
\end{equation*}
Hence, the subgroups $\widehat{\mathcal{O}}_{K,f_1}^\times$ and $\widehat{\mathcal{O}}_{K,f_2}^\times$ in $I_{K,f}$ generate evidently $\widehat{\mathcal{O}}_{K,\gcd\{f_1,f_2\}}^\times$. With the surjectivity of the universal norm residue symbol $(\cdot, K^{\mathrm{ab}}/K)$, we infer that $(\widehat{O}_{K,f_1}^\times K^\times/K^\times,K^{\mathrm{ab}}/K)$ and $(\widehat{O}_{K,f_2}^\times K^\times/K^\times,K^{\mathrm{ab}}/K)$ generate $(\widehat{O}_{K,\gcd\{f_1, f_2\}}^\times K^\times/K^\times,K^{\mathrm{ab}}/K)$.
This is equivalent to the first assertion (\ref{equation::rcf_intersection}).

For the second assertion (\ref{equation::rcf_union}), we refer to \cite[Section 3]{Allombert2015}.
\end{proof}

Finally, we recall an important lemma governing the intersections of ring class fields associated to several distinct imaginary quadratic fields. This is an essential tool in our argument.

\begin{lemma}{(\cite[Corollary 1]{Kuehne2017})} \label{lemma::ringclassfieldintersections} Let $K_1, \dots, K_r$ ($r \geq 2$) be distinct totally imaginary quadratic number fields and $f_1,\dots, f_r$ arbitrary positive integers. Write $L = K_r \prod_{i=1}^{r-1}K_i[f_i]$. Then, the Galois group $\Gal(K_r[f_r] \cap L/K_r)$ is annihilated by $2^{r+1}$.
\end{lemma}

\subsection{Explicit class number bounds} \label{subsection::classnumbers}

To each quadratic extension $K/\IQ$ there is associated a Dirichlet character $\chi_K(n)=(\disc(\mathcal{O}_K)/n)$ by means of Kronecker's symbol $(d/n)$; for details we refer to \cite[Section 5.B]{Cox1989}. Let $L(1,\chi_K)$ be the associated Dirichlet $L$-function. The Siegel-Tatuzawa Theorem \cite[Theorem 2]{Tatuzawa1951} (see also \cite{Hoffstein1980/81}) implies that there exists (at most) one imaginary quadratic field $K_\ast$ such that
\begin{equation*}
L(1,\chi_K)> (5.4 \cdot 10^{-2}) \cdot \left|\disc(\mathcal{O}_K)\right|^{-1/12}
\end{equation*}
for any other imaginary quadratic field $K \neq K_\ast$ with $\disc(\mathcal{O}_K) \geq 1.63 \cdot 10^5$. Using the class number formula \cite[(15) on p.\ 49]{Davenport2000}, we infer that
\begin{equation*}
\# \Pic(\mathcal{O}_K) = \frac{(\# \mathcal{O}_K^\times)\left|\disc(\mathcal{O}_K)\right|^{1/2}}{2\pi} L(1,\chi_K) > (1.7 \cdot 10^{-2}) \cdot \left|\disc(\mathcal{O}_K)\right|^{5/12}
\end{equation*}
whenever $\disc(\mathcal{O}_K) \geq 1.63 \cdot 10^5$. The restriction on the discriminant can be lifted by worsening the constant. In fact, the estimate
\begin{equation*}
\# \Pic(\mathcal{O}_K) > (6.7 \cdot 10^{-3}) \cdot \left|\disc(\mathcal{O}_K)\right|^{5/12}
\end{equation*}
is true for arbitrary imaginary quadratic fields $K\neq K_\ast$. Since $p^{1/6}(1-1/p)\geq 1$ for all primes $p \geq 5$, we have the following variant of (\ref{equation::classnumberformula}):
\begin{equation*}
\frac{\# \Pic(\mathcal{O}_{K,f})}{\# \Pic(\mathcal{O}_K)}=[K[f]:K[1]] \geq \frac{f^{5/6}}{9}.
\end{equation*}
This can be used to deduce the more general bound
\begin{equation} \label{equation::classnumberbound}
\# \Pic(\mathcal{O}) > (7.4 \cdot 10^{-4}) \cdot \left|\disc(\mathcal{O})\right|^{5/12}
\end{equation}
for any imaginary quadratic order $\mathcal{O}$ not contained in $K_\ast$.

\subsection{Genus number bounds} \label{subsection::genustheory} For any imaginary quadratic field $K$, we have
\begin{equation*}
\dim_{\IF_2}(\Pic(\mathcal{O})[2]) \leq 1 + 2\omega(\disc(\mathcal{O}))
\end{equation*}
(cf.\ \cite[Proposition 6.3]{Zhang2005}). It is easy to deduce from the explicit bounds on the arithmetic function $\omega(\cdot)$ in the literature (e.g., \cite[Th\'eor\`eme 11]{Robin1983}) and the fact that $\# \Pic(\mathcal{O}) = 1$ for $\disc(\mathcal{O}) \leq 11$ the weak estimate
\begin{equation*}
\dim_{\IF_2}(\Pic(\mathcal{O})[2]) \leq 1 + 3.05 \log\left|\disc(\mathcal{O})\right| \leq 4 \log\left|\disc(\mathcal{O})\right|,
\end{equation*}
which is sufficient for our purposes. In fact, we only use this observation only in the following form: For every integer $n\geq 1$, the standard estimate $e^x \geq x^n / n!$ for $x \geq 0$ and Stirling's approximation 
\begin{equation*}
n! \leq e n^{n+1/2} e^{-n}
\end{equation*}
imply that
\begin{equation} \label{equation::pic2}
\dim_{\IF_2}(\Pic(\mathcal{O})[2]) \leq 4\sqrt[n]{n!} \cdot \left|\disc(\mathcal{O})\right|^{1/n} < 4 n^2 |\disc(\mathcal{O\textsl{•}})|^{1/n}.
\end{equation}
\section{Linear Equations in Distinct Singular Moduli}
\label{section::linearequations}

This section is the heart of our article. Its purpose is to establish the following two lemmas.

\begin{lemma} \label{lemma::samefield} 
Let $a_1,a_2,\dots,a_k, b$ be algebraic numbers and assume that $a_1,\dots,a_k$ are non-zero. If $j(\tau_i)$, $i\in \{ 1,\dots, k\}$, are distinct singular moduli satisfying the linear equation
\begin{equation} \label{equation::samefield}
a_1 j(\tau_1)+a_2 j(\tau_2) + \cdots + a_k j(\tau_k) +b = 0,
\end{equation}
then
\begin{equation*}
|\Delta(\tau_i)|^{1/2} < c_1(\underline{a},b)
\end{equation*}
for all $i \in \{ 1, \dots, k \}$. In particular, (\ref{equation::samefield}) has at most finitely many solutions of this form.
\end{lemma}

\begin{lemma} \label{lemma::automorphism} Let $a_1,a_2,\dots,a_k$ be non-zero algebraic numbers. Let further $j(\tau_1),\dots,j(\tau_k)$ be distinct singular moduli and set $L=\IQ(a_1j(\tau_1)+a_2j(\tau_2)+\cdots+a_kj(\tau_k))$. For each $i \in \{ 1, \dots, k \}$, we have 
\begin{equation*}
|\Delta(\tau_i)|^{1/2}<c_2(\underline{a})
\end{equation*}
or
\begin{equation} \label{equation::numberautomorphisms}
\left[ L(\tau_i,j(\tau_i)):L(\tau_i)\right] \leq \# \{ i^\prime \in [1,k] \cap \IZ \ | \ \Delta(\tau_{i^\prime})= \Delta(\tau_i) \} \leq k.
\end{equation}
\end{lemma}

We have not stated the above constants $c_1(\underline{a},b)$ and $c_2(\underline{a})$ explicitly. The reason for this is that we prove subcases of the above lemmas in increasing generality and that the constants evolve during these generalizations. For the general case of Lemma \ref{lemma::samefield}, an explicit value for $c_1(\underline{a},b)$ is given at the end of the following proof. In Section \ref{section::andreoort}, it gives rise to the explicit bound of Theorem \ref{theorem::main}.

\begin{proof}[Proof of Lemmas \ref{lemma::samefield} and \ref{lemma::automorphism}] We prove subcases of these lemmas in increasing generality, eventually arriving at their full generality. Our strategy is the following:
\begin{enumerate}
\item[Step 1.] Let $\Delta_1, \dots, \Delta_l$ be discriminants of imaginary quadratic orders. We prove that if Lemma \ref{lemma::samefield} is true for all singular moduli $j(\tau_1), j(\tau_2), \dots, j(\tau_k)$ satisfying $\Delta(\tau_i) \in \{ \Delta_1, \dots, \Delta_l \}$, then so is Lemma \ref{lemma::automorphism} under the same restriction on the singular moduli.
\item[Step 2.] We prove Lemma \ref{lemma::samefield}, and hence Lemma \ref{lemma::automorphism}, in the case where $\Delta(\tau_1)=\cdots=\Delta(\tau_k)$.
\item[Step 3.] We prove Lemma \ref{lemma::samefield}, and hence Lemma \ref{lemma::automorphism}, in the case where all $\tau_i$ are contained in the same imaginary quadratic field. We make essential use of the part of Lemma \ref{lemma::automorphism} proven in Step 2.
\item[Step 4.] We establish Lemma \ref{lemma::samefield} in the general case by combining Lemma \ref{lemma::ringclassfieldintersections} with the part of Lemma \ref{lemma::automorphism} proven in Step 3.
\end{enumerate}

Throughout the proof, we write $f_i \in \IZ^{\geq 1}$ for the conductor of the order $\mathcal{O}(\tau_i)$. We also write $N$ (resp.\ $N_0$) for the normal closure of $\IQ(a_1,\cdots,a_k,b)$ (resp.\ $\IQ(a_1,\cdots,a_k)$) and use $H$ (resp.\ $H_0$) as an abbreviation for the (normalized affine) Weil height of $(a_1,a_2,\dots,a_k,b) \in \IA^{k+1}(\IQbar)$ (resp.\ $(a_1,a_2,\dots,a_k) \in \IA^k(\IQbar)$) as defined in Section \ref{subsection::heights}. For any algebraic number $b \in \IQbar$, we define its house
\begin{equation*}
\house{b} = \max_{\nu \in \Sigma_\infty(\IQ(b))} \left\{ |b|_\nu\right\}.
\end{equation*}
To simply our notation, we write $\log^+(\house{b})$ for $\log \max \{ 1, \house{b} \} $. Furthermore, we adopt the convention to write singular moduli always in the form $j(\tau)$ with $\tau \in \mathcal{F}$.

\textit{Step 1.} Let $j(\tau_1),\dots,j(\tau_k)$ be as stipulated in Lemma \ref{lemma::automorphism}. Assume furthermore that Lemma \ref{lemma::samefield} is proven for all singular moduli $j(\tau_1^\prime),\dots,j(\tau^\prime_{k^\prime})$ with $\{ \Delta(\tau_1^\prime),\dots, \Delta(\tau_{k^\prime}^\prime) \} \subseteq \{ \Delta(\tau_1),\dots, \Delta(\tau_k) \}$. The extension $L^\prime=L(\tau_i,j(\tau_i))/L(\tau_i)$ is Galois as $\IQ(\tau_i,j(\tau_i))/\IQ(\tau_i)$ is so. Therefore, it suffices to bound the number of automorphisms $\sigma$ of $L^\prime$ satisfying
\begin{equation} \label{equation::noninvariance2}
a_1 j(\tau_1) + \cdots + a_k j(\tau_k) = (a_1 j(\tau_1) + \cdots + a_k j(\tau_k))^\sigma.
\end{equation}
Lifting $\sigma$ to some $\widetilde{\sigma} \in \Gal(\IQbar/L(\tau_i))$ and rearranging terms, we can rewrite (\ref{equation::noninvariance2}) as a linear equation
\begin{equation} \label{equation::noninvariance3}
a_1^\prime j(\tau^\prime_1) + \cdots + a_{k^\prime}^\prime j(\tau^\prime_{k^\prime}) = 0
\end{equation} 
to which Lemma \ref{lemma::samefield} is applicable. If $j(\tau_i)$ does not appear in (\ref{equation::noninvariance3}) because of cancellation, there has to exist some $i^\prime \in \{1, \dots, k\}$ such that $j(\tau_i)^{\widetilde{\sigma}}=j(\tau_{i^\prime})$ and $a_i^{\widetilde{\sigma}}=a_{i^\prime}$. We note that $\Delta(\tau_i)=\Delta(\tau_{i^\prime})$. Since $j(\tau_i)$ generates $L^\prime$ over $L(\tau_i)$, there is at most one automorphism $\sigma \in \Gal(L^\prime/L(\tau_i))$ whose liftings $\widetilde{\sigma}$ can satisfy this for a fixed $i^\prime$. Consequently, the bound (\ref{equation::numberautomorphisms}) is valid in this case. In the case where $j(\tau_i)$ still appears in (\ref{equation::noninvariance3}), we obtain immediately $|\Delta(\tau_i)|^{1/2} < c_1(a_1^\prime, \dots, a_{k^\prime}^\prime, 0)$ by Lemma \ref{lemma::samefield}.

\textit{Step 2.} Let $(j(\tau_1),\dots,j(\tau_k))$ be a solution of (\ref{equation::samefield}) in distinct singular moduli such that $\Delta(\tau_1)=\cdots = \Delta(\tau_k)$. This means that there is an imaginary quadratic field $K$ such that $\tau_1,\dots,\tau_k \in K$ and that the conductors $f_1, \dots, f_k$ are all equal. Write $\Delta$ for the discriminant of $\mathcal{O}_K$ and $f$ for $f_1 = \dots = f_k$ so that $\Delta(\tau_i)=f^2 \Delta$. After applying a Galois automorphism to (\ref{equation::samefield}), we may assume that $\mathrm{Im}(\tau_1)=f|\Delta|^{1/2}/2$. From the explicit description in (\ref{equation::taus}), we infer then that $\mathrm{Im}(\tau_i) \leq f|\Delta|^{1/2}/4$ for $i \in \{ 2,\dots,k \}$. 
 
The bound on $f|\Delta|^{1/2}$ follows from the archimedean estimate
\begin{equation} \label{equation::inequality1}
|j(\tau_1)| = |a_1^{-1}a_2 j(\tau_2) + \cdots + a_1^{-1}a_k j(\tau_k) + a_1^{-1}b| \leq \sum_{i=2}^k |a_1^{-1}a_i| \cdot |j(\tau_i)| +|a_1^{-1}b|.
\end{equation}
In fact, we obtain
\begin{equation*}
|j(\tau_1)| \geq e^{2\pi \mathrm{Im}(\tau_1)}-2079 = e^{\pi f|\Delta|^{1/2}} - 2079
\end{equation*}
as well as
\begin{equation*}
|j(\tau_i)| \leq e^{2\pi \mathrm{Im}(\tau_i)}+2079 \leq e^{\pi f|\Delta|^{1/2}/2} + 2079, \ i \in \{ 2, \dots, k\},
\end{equation*}
by using (\ref{equation::jestimate}). Combining these estimates with (\ref{equation::inequality1}) yields
\begin{align} \label{equation::inequalitystep2}
e^{\pi f|\Delta|^{1/2}}-2079 
&\leq (k-1) H_0^{2[N_0:\IQ]}(e^{\pi f|\Delta|^{1/2}/2} + 2079)+ H_0^{[N_0:\IQ]}\house{b}.
\end{align}
It is easy to check that is not possible if $e^{\pi f|\Delta|^{1/2}/2} \geq 70k H_0^{[N_0:\IQ]} \max \{ H_0^{[N_0:\IQ]}, \house{b}^{1/2} \}$. We obtain thus Lemma \ref{lemma::samefield} in our situation with
\begin{equation} \label{equation::discriminantbound0}
c_1(\underline{a},b) = 2 [N_0:\IQ] \log(H_0) + \log^+(\house{b}) + \log(70k).
\end{equation}
Inspecting the argument of Step 1, we see that the linear equation (\ref{equation::noninvariance3}) satisfies 
\begin{equation}
H(a_1^\prime, a_2^\prime,\dots,a_{k^{\prime}}^\prime) \leq 2H_0^2 \text{ and } k^\prime \leq 2k;
\end{equation}
this implies Lemma \ref{lemma::automorphism} with 
\begin{equation*}
c_2(\underline{a}) = 4 [N_0:\IQ]\log(H_0) + 2[N_0:\IQ]\log(2)+ \log(140k).
\end{equation*}

\textit{Step 3.} Let $(j(\tau_1),\dots,j(\tau_k))$ be a solution of (\ref{equation::samefield}) in distinct singular moduli such that there exists an imaginary quadratic field $K$ and $\tau_i \in K$ ($i \in \{ 1, \dots, k \}$). Write again $\Delta$ for the discriminant of $\mathcal{O}_K$. Renaming if necessary, we may assume that $f_1 \geq \cdots \geq f_k$. Let $l \in \{ 1,\dots, k\}$ be maximal such that $f_1=f_2=\cdots=f_l$ and write $f$ for this conductor. By applying a Galois automorphism to (\ref{equation::samefield}), we may furthermore suppose that $\mathrm{Im}(\tau_1)=f|\Delta|^{1/2}/2$ without loss of generality. Note that we have $\mathrm{Im}(\tau_i) \leq f|\Delta|^{1/2}/4$ for $i \in \{ 2,\dots,l \}$ but that this may not be true for $i=l+1$.

We first prove the lemma under the additional assumption that
\begin{equation} \label{equation::assumption}
\pi (f - f_{l+1}) |\Delta|^{1/2} \geq 2 [N_0:\IQ] \log(H_0) + \log(2k).
\end{equation}
As in Step 2, we want to deduce an estimate from (\ref{equation::inequality1}). For this, we use (\ref{equation::jestimate}) and obtain
\begin{equation*}
|j(\tau_1)| \geq e^{2\pi \mathrm{Im}(\tau_1)}-2079 = e^{\pi  f|\Delta|^{1/2}} - 2079
\end{equation*}
as well as
\begin{equation*}
|j(\tau_i)| \leq e^{2\pi \mathrm{Im}(\tau_i)}+2079 \leq e^{\pi f|\Delta|^{1/2}/2} + 2079, \ i \in \{ 2, \dots, l\},
\end{equation*}
and
\begin{equation*}
|j(\tau_i)| \leq e^{2\pi \mathrm{Im}(\tau_i)}+2079 \leq e^{\pi  f_{l+1}|\Delta|^{1/2}} + 2079, \ i \in \{ l+1, \dots, k\}.
\end{equation*}
Combining these estimates with (\ref{equation::inequality1}) yields
\begin{equation} \label{equation::boundstep3_1}
e^{\pi f|\Delta|^{1/2}}-2079 \leq (k-1) H_0^{2[N_0:\IQ]} \left( e^{\pi f|\Delta|^{1/2}} \max \{ e^{- \pi f|\Delta|^{1/2}/2},  e^{\pi (f_{l+1}-f)|\Delta|^{1/2}}\} + 2079 \right) + H_0^{[N_0:\IQ]}\house{b}.
\end{equation}
If $e^{-\pi f|\Delta|^{1/2}/2} \geq e^{\pi (f_{l+1}-f)|\Delta|^{1/2}}$, this inequality coincides with (\ref{equation::inequalitystep2}). In this case, there is thus nothing to prove as long as
\begin{equation} \label{equation::lowerboundonc1}
c_1(\underline{a},b) \geq 2 [N_0:\IQ] \log(H_0) + \log^+(\house{b}) + \log(70k).
\end{equation}
If $e^{-\pi f|\Delta|^{1/2}/2} < e^{\pi (f_{l+1}-f)|\Delta|^{1/2}}$, combining inequality (\ref{equation::boundstep3_1}) with our assumption (\ref{equation::assumption}) yields
\begin{equation*}
e^{\pi f|\Delta|^{1/2}} < 4158(k H_0^{2[N:\IQ]}+1) + 2 H_0^{[N_0:\IQ]} \house{b}< 8400 k H_0^{[N_0:\IQ]} \max \{ H_0^{[N_0:\IQ]} , \house{b} \}.
\end{equation*}
Comparing with (\ref{equation::lowerboundonc1}) above, we see that 
\begin{equation} \label{equation::discriminantbound1}
c_1(\underline{a},b) = 2 [N_0:\IQ] \log(H_0) + \log^+(\house{b})  + \log (8400k) 
\end{equation}
is an eligible choice if (\ref{equation::assumption}) is satisfied.

It remains to reduce the general case to the one already considered. Assume that
\begin{equation*}
\pi (f-f_{l+1}) |\Delta|^{1/2} < 2[N_0:\IQ] \log(H_0) + \log(2k).
\end{equation*}
In other words, $f_{l+1}$ is contained in the open interval 
\begin{equation*}
\left(f- |\Delta|^{-1/2} \pi^{-1}(2[N_0:\IQ] \log(H_0) + \log(2k)),f\right)
\end{equation*}
so that
\begin{equation*}
\gcd  \{ f, f_{l+1} \} < |\Delta|^{-1/2}\pi^{-1}(2[N_0:\IQ] \log(H_0) + \log(2k)).
\end{equation*}
Hence,
\begin{equation} \label{equation::gcdbound}
\lcm \{ f, f_{l+1} \} / f_{l+1} > f |\Delta|^{1/2}\pi\left(2[N_0:\IQ] \log(H_0) + \log(2k)\right)^{-1}.
\end{equation}

Our aim is to give a \textit{new} linear equation 
\begin{equation} \label{equation::newlinearequation}
a_1^{(1)} j(\tau_1^{(1)}) + \cdots + a_{k^{(1)}}^{(1)} j(\tau_{k^{(1)}}^{(1)}) + b^{(1)} = 0,
\end{equation}
in distinct singular moduli $j(\tau_1^{(1)}), \dots, j(\tau_{k^{(1)}}^{(1)})$ with
\begin{equation} \label{equation::discriminantcondition}
f^2 \Delta \in \{ \Delta(\tau_1^{(1)}), \dots, \Delta(\tau_{k^{(1)}}^{(1)}) \} \subseteq \{f_1^2 \Delta,\dots, f_k^2 \Delta \} \setminus \{ f_{l+1}^2 \Delta \}.
\end{equation}
(This means that we remove not only $j(\tau_{l+1})$ but all other singular moduli having discriminant $f_{l+1}^2 \Delta$ as well.) We obtain such an equation as the difference between (\ref{equation::samefield}) and one of its Galois conjugates. For this, we prove that that there exists an element 
\begin{equation*}
\sigma \in \Gal(N_0 \cdot K[f] \cdot K[f_{l+1}]/ N_0 \cdot K[f_{l+1}])=: G
\end{equation*}
such that
\begin{equation} \label{equation::noninvariance}
(a_1 j(\tau_1) + \cdots + a_l j(\tau_l))^\sigma \neq a_1 j(\tau_1) + \cdots + a_l j(\tau_l).
\end{equation}
Our new linear equation (\ref{equation::newlinearequation}) arises then from lifting $\sigma$ to some $\widetilde{\sigma} \in \Gal(\IQbar/N_0 \cdot K[f_{l+1}])$ and regrouping the terms in 
\begin{equation*}
(a_1 j(\tau_1)+a_2 j(\tau_2) + \cdots + a_k j(\tau_k) + b) - (a_1 j(\tau_1)+a_2 j(\tau_2) + \cdots + a_k j(\tau_k) + b)^{\widetilde{\sigma}} = 0.
\end{equation*}
By construction, the condition (\ref{equation::discriminantcondition}) is evidently verified.

On the one hand, restriction induces an injection $G \hookrightarrow \Gal(N_0 \cdot K[f]/N_0 \cdot K)$. By the part of Lemma \ref{lemma::automorphism} that is established in Step 1, there are at most $l \leq k$ elements in $\Gal(N_0 \cdot K[f]/ N_0 \cdot K)$ violating (\ref{equation::noninvariance}) unless
\begin{equation} \label{equation::discriminantbound2}
f|\Delta|^{1/2} < 4[N_0:\IQ]\log(H_0)+ 2[N_0:\IQ]\log(2)+ \log(140k),
\end{equation}
which we can exclude from the outset by choosing $c_1(\underline{a},b)$ sufficiently large. On the other hand, the diagram
\begin{equation*}
\xymatrix{
& N_0 \cdot K[\lcm \{ f, f_{l+1} \}] \ar@{-}[ld] \ar@{-}[rd] & \\
N_0 \cdot K[f_{l+1}] \ar@{-}[rd] & & K[\lcm \{ f, f_{l+1} \} ] \ar@{-}[ld] \\
& K[f_{l+1}]
}
\end{equation*}
of field extensions in combination with (\ref{equation::classnumberformula}) and (\ref{equation::rcf_union}) shows that
\begin{equation*}
\# G \geq \frac{\left[K[\lcm \{ f, f_{l+1} \}]:K[f_{l+1}]\right]}{3[N_0:\IQ]} \geq \frac{\sqrt{6}}{36[N_0:\IQ]} \cdot \left( \frac{\lcm \{ f, f_{l+1} \}}{f_{l+1}}\right)^{1/2}.
\end{equation*}
Hence, there is either an element $\sigma \in G$ satisfying (\ref{equation::noninvariance}) or
\begin{equation*}
\lcm\{f,f_{l+1}\}/f_{l+1}\leq 216 k^2 [N_0:\IQ]^2.
\end{equation*}
However, plugging this inequality together with (\ref{equation::gcdbound}) yields
\begin{equation}  \label{equation::discriminantbound3}
f|\Delta|^{1/2} < 138 k^2 [N_0:\IQ]^3 \log(H_0) + 69 k^2 \log(2k) [N_0:\IQ]^2;
\end{equation}
thus, we can exclude this case likewise by choosing $c_1(\underline{a},b)$ sufficiently large. In the sequel, we suppose this is the case so that we have an automorphism $\sigma \in G$ with (\ref{equation::noninvariance}) at our disposal.

As described above, this allows us to obtain a new linear equation (\ref{equation::newlinearequation}) satisfying (\ref{equation::discriminantcondition}). It is clear that (\ref{equation::newlinearequation}) is a (homogeneous) linear equation in $k^{(1)} \leq 2k$ distinct singular moduli and that
\begin{equation*}
H(a_1^{(1)},\dots,a_{k^{(1)}}^{(1)}) \leq 2 H_0^2
\end{equation*}
as well as $\house{b^{(1)}} \leq 2\house{b}$. In general, the new equation (\ref{equation::newlinearequation}) does not need to fall within the scope of linear equations we can already deal with (i.e., those satisfying (\ref{equation::assumption}) or those already considered in Step 2 above). In any case, we can repeat the above procedure until we end up with a linear equation satisfying its respective version of (\ref{equation::assumption}) or for which we can invoke Step 2. Indeed, by (\ref{equation::discriminantcondition}) the number of discriminants associated with a singular moduli appearing in the linear equation drops by one at each step. This can be repeated at most $k-1$ times until we obtain a linear equation of the shape treated in Step 2. At the $i$-th iteration, we obtain a \textit{non-trivial} linear equation 
\begin{equation*}
a_1^{(i)} j(\tau_1^{(i)}) + \cdots + a_{k^{(i)}}^{(i)} j(\tau_{k^{(i)}}^{(i)}) + b^{(i)} = 0,
\end{equation*}
whose number of coefficients $k^{(i)}$ and height $H^{(i)}_0$ are bounded by 
\begin{equation*}
k^{(i)}\leq 2 k^{(i-1)} \leq \cdots \leq 2^{i} k
\end{equation*}
and
\begin{equation*}
H^{(i)}_0 \leq 2 (H^{(i-1)}_0)^2 \leq \cdots \leq 2^{2^{i}-1}H^{2^{i}}_0.
\end{equation*}
Additionally, we have $\house{b^{(i)}} \leq 2^{i} \house{b}$. To obtain an explicit value for $c_1(\underline{a},b)$, it suffices to consider (\ref{equation::discriminantbound0}), (\ref{equation::discriminantbound1}), (\ref{equation::discriminantbound2}), and (\ref{equation::discriminantbound3}) with $(k,H_0,\house{b})$ replaced by $(2^{k-1}k,2^{2^{k-1}-1}H^{2^{k-1}}_0,2^{k-1}\house{b})$.\footnote{Before substituting, note that the left-hand sides in (\ref{equation::discriminantbound0}), (\ref{equation::discriminantbound1}), (\ref{equation::discriminantbound2}), and (\ref{equation::discriminantbound3}) are majorized by
\begin{equation*}
138 k^2 [N_0:\IQ]^3 \log(H_0) + \log^+(\house{b}) + 69 k^2 \log(2k) [N_0:\IQ]^2.
\end{equation*}
} A quick computation yields that
\begin{equation} \label{equation::discriminantboundstep4_2}
c_1(\underline{a},b) = 
18 k^2 8^{k} [N_0:\IQ]^3 \log(H_0) + \log^+(\house{b}) + 21 k^3 8^k [N_0:\IQ]^3.
\end{equation}
is an appropriate choice. By Step 1, Lemma \ref{lemma::automorphism} is true with
\begin{equation} \label{equation::discriminantboundstep4}
c_2(\underline{a}) = 144 k^264^k[N_0:\IQ]^3 \log(H_0) + 218 k^3 64^{k} [N_0:\IQ]^3
\end{equation}
if all singular moduli are associated with the same imaginary quadratic field $K$.

\textit{Step 4.} Finally, we consider a general solution of (\ref{equation::samefield}) in distinct singular moduli $j(\tau_1),\dots,j(\tau_k)$. Without loss of generality, we may assume that there are integers $k_1,\dots,k_{r-1}$ satisfying $0=k_0 < k_1 < k_2 < \cdots < k_r = k$ and imaginary quadratic fields $K_1,\dots,K_r$ such that
\begin{align*}
\IQ(\tau_1)=\IQ(\tau_2)= &\cdots = \IQ(\tau_{k_1}) = K_1, \\
\IQ(\tau_{k_1+1})=\IQ(\tau_{k_1+2})= &\cdots = \IQ(\tau_{k_2}) = K_2, \\
&\dots \\
\IQ(\tau_{k_{r-1}+1})=\IQ(\tau_{k_{r-1}+2}) = &\cdots = \IQ(\tau_{k}) = K_r.
\end{align*}
To simplify notation, we write respectively $a_i^{(j)}$, $\tau_i^{(j)}$, $f_i^{(j)}$, and $l^{(j)}$ for $a_{k_{j-1}+i}$, $\tau_{k_{j-1}+i}$, $f_{k_{j-1}+i}$, and $k_j-k_{j-1}$.\footnote{The reader is warned not to confuse the index $j$ with the $j$-invariant. The latter is only used as part of the expression $j(\cdot)$ so that there should be no confusion.} We define
\begin{equation*}
L_j=K_j(a^{(j)}_{1}j(\tau^{(j)}_{1}) + \cdots + a^{(j)}_{l^{(j)}}j(\tau^{(j)}_{l^{(j)}})),
\end{equation*}
and
\begin{equation*}
M_j=K_j(j(\tau^{(j)}_{1}),j(\tau^{(j)}_{2}),\dots,j(\tau^{(j)}_{l^{(j)}}))
\end{equation*}
for all $j \in \{ 1, \dots, r \}$. The extension $K_j(j(\tau_i^{(j)}))/\IQ$ is Galois (see Section \ref{subsection::ringclassfields}) and so is $M_j/\IQ$. We denote by $\overline{L}_j$ the normal closure of $L_j$ over $K_j$. If $r=1$, there is nothing to prove; for we only have to ensure that
\begin{equation*}
c_1(a,\underline{b}) \geq
18 k^2 8^{k} [N_0:\IQ]^3 \log(H_0) + \log^+(\house{b}) + 21 k^3 8^k [N_0:\IQ]^3.
\end{equation*}
because of Step 3 above.

Recall the (possibly non-existent) exceptional field $K_\ast$ from Section \ref{subsection::classnumbers}.
In the remaining case where $r>1$, our first goal is to bound $|\Delta(\tau_i^{(j)})|$ whenever $K_j \neq K_\ast$. For this, let $c_3=c_3(k,H_0,[N_0:\IQ])$ denote the right-hand side of (\ref{equation::discriminantboundstep4}), and assume that we have both $|\Delta(\tau_i^{(j)})|^{1/2} \geq c_3$ and $K_j \neq K_\ast$ simultaneously for some $j \in \{ 1,\dots, r \}$ and $i \in \{ 1, \dots, l^{(j)}\}$. By virtue of the special case of Lemma \ref{lemma::automorphism} proven in Step 3, we know that
\begin{equation*}
[\overline{L}_j(j(\tau_i^{(j)})):\overline{L}_j] \leq [L_j(j(\tau_i^{(j)})):L_j] \leq l^{(j)} \leq k.
\end{equation*}
Write $M_j^\prime=\prod_{\substack{ 1\leq j^\prime \leq r \\ j^\prime \neq j }}M_{j^\prime}$ and consider the following diagram of field extensions:
\begin{equation*} 
\xymatrix{
& \overline{L}_j(j(\tau_i^{(j)})) \cdot \left( NM_j \cap N K_j M_j^\prime \right)=:F_0 \ar@{-}[dl] \ar@{-}[dr] & \\
\overline{L}_j(j(\tau_i^{(j)})) \ar@{-}[d] \ar@{-}[dr] & & NM_j \cap N K_j M_j^\prime=:F_1 \ar@{-}[dl] \ar@{-}[d]\\
K_j(j(\tau_i^{(j)})) \ar@{-}[dr] & \overline{L}_j \ar@{-}[d] & M_j \cap K_j M_j^\prime=:F_2 \ar@{-}[dl] \\
& K_j & & 
}
\end{equation*}
(Here, $\overline{L}_j$ is in $F_1$ because $L_j \subseteq F_1$ and $F_1/K_j$ is Galois.)
We claim that the abelian group $\Gal(K_j(j(\tau_i^{(j)}))/K_j)$ is an extension of a finite group of exponent $2^{r+1}$ by a group of order $\leq k [N:\IQ]^2$. First of all, note that $F_0/K_j$ is a Galois extension, being the composite of several extensions of $K_j$ that are evidently Galois. Next, we note that $\Gal(F_0/F_2)$ is a normal subgroup of $\Gal(F_0/K_j)$ with at most $k [N:\IQ]^2$ elements. In fact,
\begin{equation*}
[F_0:F_1] \leq [\overline{L}_j(j(\tau_i^{(j)})):\overline{L}_j] \leq k
\end{equation*}
and Lemma \ref{lemma::galoistheory} below implies that
\begin{equation*}
[F_1:F_2] \leq [N:\IQ]^2.
\end{equation*}
In addition, we derive from Lemma \ref{lemma::ringclassfieldintersections} that $\Gal(F_2/K_j)$ is annihilated by $2^{r+1}$. As the group $\Gal(K_j(j(\tau_i^{(j)}))/K_j)$ is a quotient of $\Gal(F_0/K_j)$, the claim follows directly from Lemma \ref{lemma::grouptheory} below.

From Section \ref{subsection::ringclassfields}, we know that $\Gal(K_j(j(\tau_i^{(j)}))/K_j) = \Pic(\mathcal{O}(\tau_i^{(j)}))$ so that the above yields
\begin{equation*}
\# \Pic(\mathcal{O}(\tau_i^{(j)})) \leq k [N:\IQ]^2 \dim_{\IF_2}(\Pic(\mathcal{O}(\tau_i^{(j)}))[2])^{r+1}.
\footnote{We use here tacitly that $\# G[2] \geq \# H[2]$ for any surjective homomorphism $G \twoheadrightarrow H$ of finite abelian groups. Since $\# G[2] = \# (G \otimes \IF_2)$ this follows directly from the right-exactness of $(\cdot) \otimes \IF_2$.}
\end{equation*}
Using (\ref{equation::pic2}) with $n=6(r+1)$ and $r \leq k$, we deduce from this
\begin{equation*}
\# \Pic(\mathcal{O}(\tau_i^{(j)})) \leq 144^{k+1} (k+1)^{2k+3} [N:\IQ]^2 |\Delta(\tau_i^{(j)})|^{1/6}.
\end{equation*}
In combination with (\ref{equation::classnumberbound}), we obtain
\begin{equation*}
(7.4 \cdot 10^{-4}) \cdot |\Delta(\tau_i^{(j)})|^{5/12} \leq \# \Pic(\mathcal{O}(\tau_i^{(j)})) \leq 144^{k+1} (k+1)^{2k+3} [N:\IQ]^2 |\Delta(\tau_i^{(j)})|^{1/6}.
\end{equation*}
and hence
\begin{equation*}
|\Delta(\tau_i^{(j)})|^{1/2} < (3.8 \cdot 10^{10}) (2.1 \cdot 10^4)^{k} (k+1)^{4k+6} [N:\IQ]^4.
\end{equation*}

It remains to bound $|\Delta(\tau_i^{(j)})|^{1/2}$ in case $K_j = K_\ast$. To ease notation, let us assume that $K_1 = K_\ast$. (If no $K_i$ equals the exceptional field $K_\ast$, we are already done at this point.) We rewrite the original linear equation (\ref{equation::samefield}) as
\begin{equation} \label{equation::rewrittenequation}
a_1 j(\tau_1)+a_2 j(\tau_2) + \cdots + a_{k_1} j(\tau_{k_1}) + b^\prime = 0,
\end{equation}
with
\begin{equation*}
b^\prime = b + a_{k_1+1} j(\tau_{k_1+1}) + \cdots + a_{k} j(\tau_{k}).
\end{equation*}
In words, we put all the singular moduli with CM-field other than $K_\ast$ into the constant term. From above, we know that
\begin{equation*}
|\Delta(\tau_i)|^{1/2} < 144 k^264^k[N:\IQ]^3 \log(H_0) + (3.8 \cdot 10^{10}) (2.1 \cdot 10^4)^{k} (k+1)^{4k+6} [N:\IQ]^4
\end{equation*}
for each $i \in [k_1+1,k] \cap \IZ$. With the estimate (\ref{equation::jhouse}), we have
\begin{equation*}
\house{b^\prime} < \house{b} + 11 k H^{[N:\IQ]} \max_{k_1+1 \leq i \leq k}\{\exp(\pi|\Delta(\tau_i)|^{1/2})\},
\end{equation*}
and hence
\begin{multline} \label{equation::housebound}
\log^+(\house{b^\prime}) < 460 k^264^k[N:\IQ]^3 \log(H_0) \\ + \log^+(\house{b}) + (1.2 \cdot 10^{11}) (2.1 \cdot 10^4)^{k} (k+1)^{4k+6} [N:\IQ]^4
\end{multline}
Since all singular moduli in the modified equation (\ref{equation::rewrittenequation}) are associated with the same CM-field, namely $K_\ast$, our Step 3 yields a bound on the remaining $|\Delta_i|^{1/2}$, $i \in \{1,\dots, k_1 \}$. In fact, we only have to plug (\ref{equation::housebound}) into (\ref{equation::discriminantboundstep4_2}). It is then easy to see that
\begin{equation} \label{equation::finalresult}
c_1(\underline{a},b) < 480 k^264^k[N:\IQ]^3 \log(H) + (1.3 \cdot 10^{11}) (2.1 \cdot 10^4)^{k} (k+1)^{4k+6} [N:\IQ]^4
\end{equation}
is an admissible choice.
\end{proof}

We use elementary Galois theory to produce a ready-to-use lemma for the above proof.

\begin{lemma} \label{lemma::galoistheory} Let $F_1$, $F_2$, $N$ be finite Galois extensions of a common base field $k$. Then,
\begin{equation*}
[NF_1 \cap NF_2 : N(F_1 \cap F_2)] \leq \min \{ [NF_1:F_1], [NF_2:F_2] \} \leq [N : k].
\end{equation*}
\end{lemma}

\begin{proof} This follows from a repeated use of \cite[Theorem VI.1.12]{Lang2002}. In fact, it yields
\begin{equation*}
[N F_1 \cap N F_2 : F_1 \cap N F_2] = [(N F_1 \cap N F_2) \cdot F_1 : F_1] \leq [N F_1 : F_1]
\end{equation*}
and
\begin{equation*}
[F_1 \cap N F_2 : F_1 \cap F_2] =  [(F_1 \cap N F_2) \cdot F_2 : F_2] \leq [N F_2 : F_2].
\end{equation*}
In addition, \cite[Corollary VI.1.13]{Lang2002} gives $[N F_i: F_i] \leq [N : k]$ ($i \in \{1,2\}$). It also implies
\begin{equation*}
\max \{ [NF_1:F_1], [NF_2:F_2] \} \leq [N(F_1 \cap F_2) : (F_1 \cap F_2)].
\end{equation*}
Combining all these inequalities, we obtain
\begin{align*}
[(NF_1 \cap NF_2) : N(F_1 \cap F_2)] 
&= \frac{[(NF_1 \cap NF_2) : (F_1 \cap F_2)]}{[N(F_1 \cap F_2): (F_1 \cap F_2)]} \\
&\leq \frac{[NF_1 : F_1][NF_2 : F_2]}{\max \{ [NF_1:F_1], [NF_2:F_2] \}}\\
&\leq \min \{ [NF_1:F_1], [NF_2:F_2] \}.
\end{align*}
\end{proof}

We record a very simple group-theoretic lemma.

\begin{lemma} \label{lemma::grouptheory} Let $G$ be a finite group with a normal subgroup $N$ such that $G/N$ is annihilated by $n$. Any quotient $G^\prime$ of $G$ has likewise a normal subgroup $N^\prime$ of size $\# N^\prime \leq \# N$ such that $G^\prime/N^\prime$ is annihilated by $n$. 
\end{lemma}

\begin{proof} Let $\pi: G \twoheadrightarrow G^\prime$ be the quotient homomorphism. The image $N^\prime = \pi(N)$ is again a normal subgroup as $\pi$ is surjective. Furthermore, the composite $G \rightarrow G^\prime \rightarrow G^\prime/\pi(N)$ is a surjection, whose kernel includes $N$. Hence, it factors through $G/N$ and $G^\prime/\pi(N)$ is annihilated by $n$.
\end{proof}

\section{Proof of Theorem \ref{theorem::main}}
\label{section::andreoort}

Let $P=(j(\tau_1),\dots,j(\tau_n)) \in (L \setminus Z^{\mathrm{sp}})(\IQbar)$ be a special point. Without loss of generality, we may assume that there are integers $n_1,\dots,n_{r-1}$ satisfying $0=n_0 < n_1 < n_2 < \cdots < n_r = n$ such that
\begin{align*}
j(\tau_1)=j(\tau_2)= &\cdots = j(\tau_{n_1}), \\
j(\tau_{n_1+1})=j(\tau_{n_1+2})= &\cdots = j(\tau_{n_2}),  \\
&\cdots \\
j(\tau_{n_{r-1}+1})=j(\tau_{n_{r-1}+2}) = &\cdots = j(\tau_{n}),
\end{align*}
and $j(\tau_{n_i})\neq j(\tau_{n_j})$ for all $i,j \in \{ 1, \dots, r\}$ such that $i \neq j$. For each $i \in \{ 1, \dots, r\}$, we write
\begin{equation*}
Z_i = \mathrm{V}(z_{n_{i-1}+1}=\cdots=z_{n_i}).
\end{equation*}
(This means $Z_i = Y(1)^n$ if $n_{i}=n_{i-1}+1$.)
Consider the special subvariety
\begin{equation*}
Z = Z_1 \cap Z_2 \cap \cdots \cap Z_r \subseteq Y(1)^n,
\end{equation*}
and form the intersection $L_1 = L \cap Z$. To ease notation, we write $n^\prime = n+1$ and $r^\prime = r + 1$ in the sequel. Recall from Section \ref{subsection::heights} that we can associate with each (affine) linear subvariety $V \subseteq \IA_{N}^k$ a linear \textit{subspace} $V^\prime = \pi^{-1}(V^h)\subseteq \IA_{N}^{k+1}$. Since $L^\prime_1 = L^\prime \cap Z^\prime$ and $Z^\prime = Z_1^\prime \cap Z_2^\prime \cap \cdots \cap Z_r^\prime$ in $\IA_N^{n^\prime}$, we have 
\begin{equation} \label{equation::linearheightbound2}
H(L_1) = H(L_1^\prime) \leq H(L^\prime) H(Z^\prime) \leq H(L^\prime) \prod_{i=1}^{r} H(Z_i^\prime) = H(L) \prod_{i=1}^rH(Z_i)
\end{equation}
by \cite[Theorem 2.8.13]{Bombieri2006}. The annihilator $Z_i^\perp \subseteq (N^\vee)^n$ of $Z_i$ has basis
\begin{equation*}
z_{n_{i-1}+1}-z_{n_{i-1}+2}, \dots, z_{n_{i-1}+1}-z_{n_i}. 
\end{equation*}
With respect to the standard basis on $(N^\vee)^n = N^n$, we can express this basis in terms of a sparse $(n \times (n_i-n_{i-1}-1))$-matrix
\begin{equation*}
\begin{pmatrix}
0 & 0 & \cdots & 0 \\
\cdots & \cdots & \cdots & \cdots \\
0 & 0 & \cdots & 0 \\
1  &  1 & \cdots & 1 \\
-1 &  0 & \cdots & 0 \\
0  &  -1 & \cdots & 0 \\
\cdots & \cdots & \cdots & \cdots \\
0 & 0 & \cdots & -1 \\
0 & 0 & \cdots & 0 \\
\cdots & \cdots & \cdots & \cdots \\
0 & 0 & \cdots & 0
\end{pmatrix}.
\end{equation*}
Using the explicit formula given in (\ref{equation::heightformula}) and \cite[Proposition 2.8.10]{Bombieri2006}, we obtain the bound 
\begin{equation} \label{equation::linearheightbound1}
H(Z_i)=H(Z_i^\perp)\leq 2^{n_i-n_{i-1}-1}(n_i-n_{i-1})^{1/2};
\end{equation}
it suffices to note that the above matrix has at most $n_i-n_{i-1}$ minors of order $(n_i - n_{i-1}-1)$ with non-vanishing discriminant and that the discriminant of each of these minors is an integer of size at most $2^{n_i-n_{i-1}-1}$. We arrive hence at
\begin{equation*}
\prod_{i=1}^r H(Z_i) \leq 2^n\prod_{i=1}^r(n_i-n_{i-1})^{1/2}\leq 2^n \left(\frac{n}{r}\right)^{r/2} \leq 2^n \exp(e^{-1} n/2) < 3^n.
\end{equation*}
With (\ref{equation::linearheightbound2}), this gives $H(L_1) < 3^n H(L)$. If $L_1 = Z$, then $P$ would be contained in $L^{\mathrm{sp}}(\IQbar)$. We exclude this case in the following. Consider the projection
\begin{equation*}
\varpi|_Z: Z \subseteq Y(1)^n \longrightarrow Y(1)^r, \ (z_1, z_2, \dots, z_n) \longmapsto (z_{n_1},z_{n_2}, \dots, z_{n_r}).
\end{equation*}
Since the restriction $\varpi|_Z$ is an isomorphism, the image $L_2=\varpi(L_1)$ is a proper linear subspace of $\varpi(Z) = Y(1)^r$. If furthermore
\begin{equation*}
A=
\begin{pmatrix}
a_{11} & a_{12} & \cdots & a_{1l} \\
a_{21} & a_{22} & \cdots & a_{2l} \\
\cdots & \cdots & \cdots & \cdots \\
a_{r^\prime 1} & a_{r^\prime 2} & \cdots & a_{r^\prime l}
\end{pmatrix}
\end{equation*}
is a $(r^\prime \times l)$-matrix whose columns form a basis of $L_2^\prime \subset \IQbar^{r^\prime}$, the $r^\prime$-th coordinate being associated with the (homogenized) degree zero part, then the columns of the $(n^\prime \times l)$-matrix
\begin{equation*}
B=\begin{pmatrix}
\left.
\begin{matrix}
1 \\
\vdots \\
1
\end{matrix}
\right\}
(n_1-n_0) & & \\
& \left.
\begin{matrix}
1 \\
\vdots \\
1
\end{matrix}
\right\}
(n_2-n_1)
& \\
& & \ddots \\
& & & \left.
\begin{matrix}
1 \\
\vdots \\
1
\end{matrix}
\right\}
(n_r-n_{r-1}) \\
& & & & 1
\end{pmatrix}
\begin{pmatrix}
a_{11} & a_{12} & \cdots & a_{1l} \\
a_{21} & a_{22} & \cdots & a_{2l} \\
\cdots & \cdots & \cdots & \cdots \\
a_{r^\prime 1} & a_{r^\prime 2} & \cdots & a_{r^\prime l}
\end{pmatrix}
\end{equation*}
form a basis of $L_1^\prime \subset \IQbar^{n^\prime}$. For each subset $I \subset \{ 1, \dots, r^\prime \}$ of cardinality $l$, we denote by $A_I$ the minor
\begin{equation*}
\begin{pmatrix}
a_{i_11} & a_{i_12} & \cdots & a_{i_1l} \\
a_{i_21} & a_{i_22} & \cdots & a_{i_2l} \\
\cdots & \cdots & \cdots & \cdots \\
a_{i_l1} & a_{i_l2} & \cdots & a_{i_ll}
\end{pmatrix}
\end{equation*}
of order $l$. Using again (\ref{equation::heightformula}), we have
\begin{equation*}
H(L_1)= 
\prod_{\nu \in \Sigma_f(N)} \left( \max_I \left\lbrace \left\vert\det(A_I)\right\vert_\nu\right\rbrace \right)
\prod_{\nu \in \Sigma_\infty(N)} \left( \sum_I \left(\prod_{j=1}^{l}(n_{i_j}-n_{i_{j-1}}) \right)\left\vert\det(A_I)\right\vert_\nu^2
\right)^{1/2}
\end{equation*}
where $I = \{ i_1, i_2, \dots, i_l\}$ runs over all subsets $I \subset \{ 1,\dots, r \}$ of cardinality $l$. Similarly, we have
\begin{equation*}
H(L_2)= 
\prod_{\nu \in \Sigma_f(N)} \left( \max_I \left\lbrace \left\vert\det(A_I)\right\vert_\nu\right\rbrace \right)
\prod_{\nu \in \Sigma_\infty(N)} \left( \sum_I\left\vert\det(A_I)\right\vert_\nu^2
\right)^{1/2}.
\end{equation*}
It is hence evident that $H(L_2) \leq H(L_1)$. In Section \ref{subsection::heights}, it is noted that this means that the point $\pi(P) = (j(\tau_{n_1}),j(\tau_{n_2}),\dots,j(\tau_{n_r})) \in L_2(\IQbar)$ has to satisfy a non-trivial linear equation
\begin{equation*}
a_1 j(\tau_{n_1}) + a_2 j(\tau_{n_2}) + \cdots + a_{n_r} j(\tau_{n_r}) + b = 0
\end{equation*}
with $a_1,a_2,\cdots,a_{n_r},b \in N$ and
\begin{equation*}
H(a_1,a_2,\cdots,a_{n_r},b) \leq H(L_2) \leq H(L_1) <3^n H(L).
\end{equation*}
Our construction is such that the singular moduli $j(\tau_{n_i})$, $1\leq i \leq r$, are pairwise distinct so that we can apply Lemma \ref{lemma::samefield}. With the constant $c_1(\underline{a},b)$ as given in (\ref{equation::finalresult}), we conclude that
\begin{align*}
|\Delta(\tau_i)|^{1/2}  < 480 k^264^k[N:\IQ]^3 \log(H) + (1.4 \cdot 10^{11}) (2.1 \cdot 10^4)^{k} (k+1)^{4k+6} [N:\IQ]^4.
\end{align*}
This completes the proof of our main theorem.

\bibliographystyle{plain}
\bibliography{references}

\end{document}